\title{Slicing and dicing polytopes}
\author{Patrik Nor\'en}
\documentclass{article}
\usepackage{amsmath,amsthm,amssymb,graphicx}
\begin{document}
\newtheorem{theorem}{Theorem}[section]
\newtheorem{definition}[theorem]{Definition}
\newtheorem{lemma}[theorem]{Lemma}
\newtheorem{proposition}[theorem]{Proposition}
\newtheorem{corollary}[theorem]{Corollary}
\newtheorem{example}[theorem]{Example}
\maketitle

\begin{abstract}
Using tropical convexity Dochtermann, Fink, and Sanyal proved that regular fine mixed subdivisions of Minkowski sums of simplices support minimal cellular resolutions. They asked if the regularity condition can be removed. We give an affirmative answer by a different method. A new easily checked sufficient condition for a subdivided polytope to support a cellular resolution is proved. The main tool used is discrete Morse theory.
\end{abstract}

\section{Introduction}

Subdivisions of polytopes often support cellular resolutions. Dochtermann, Joswig, and Sanyal~\cite{DJS} showed that regular mixed subdivisions of $n\Delta_k$ supports minimal cellular resolutions. Dochtermann, Fink, and Sanyal~\cite{DFS} recently generalized this, by showing that regular mixed subdivisions of Minkowski sums of simplices support minimal cellular resolutions. They asked if the regularity condition can be removed.

We show that the condition can be removed in the fine mixed subdivision case by introducing the concepts of diced and sharp polytopes. The main result, Theorem~\ref{thm:maint}, is that subdivisions of a diced polytope into sharp cells always give cellular resolutions. It is immediate that a Minkowski sum of simplices is diced, and that cells in the fine mixed subdivision are sharp.

A less developed version of the geometric methods of this paper was used by Engstr\"om and Nor\'en~\cite{EN} to reveal the fine structure of Betti numbers of powers of ideals in some classes. Based on that data Engstr\"om made a conjecture~\cite{E} that was proved by Mayes-Tang~\cite {M}. Erman and Sam~\cite{ES} give a good survey of the area, which contains open problems and questions that could be attacked with the methods of this paper. The limits of algebraic discrete Morse theory in a related setup to this paper was studied by Nor\'en in~\cite{N}.

In Section~\ref{Cell} the basics of cellular resolutions and discrete Morse theory are reviewed, and in Section~\ref{Poly} the concepts of diced and sharp polytopes are introduced. Finally in Section~\ref{Main}, it is showed that a subdivisions of a diced polytope into sharp cells give a cellular resolution.

\section{Cellular resolutions and Morse theory}\label{Cell}

The theory of cellular resolutions was introduced by Bayer and Sturmfels~\cite{BS}. It provides a way to obtain free resolutions of monomial ideals from cell complexes.

The cell complexes can be assumed to be CW-complexes, and cells are closed unless otherwise specified. The set of vertices, that is, the  zero dimensional cells of a cell complex $X$, is denoted by $V(X)$. This notation is also used for polytopes and cells in general. The vertex set of a polytope $P$ is $V(P)$ and of a cell $\sigma$ it is $V(\sigma)$.

\begin{definition}
A \emph{labeled cell complex} is a cell complex $X$ together with a map $\ell$ from the set of cells of $X$ to the monic monomials in $\mathbb{K}[x_1,\ldots,x_n]$. The map $\ell$ has to satisfy $\ell(\sigma)=\mathrm{lcm}\{\ell(v)\mid v\in V(\sigma)\}$ for all cells $\sigma\in X$.
\end{definition}

\begin{example}\label{ex:simp}
If the vertices of the standard simplex $\Delta_n=\mathrm{conv}\{\mathbf{e}_i\mid i\in [n]\}$ are labeled by $\ell(\mathbf{e}_i)=x_i$, then the complex becomes labeled and the label of the face $\mathrm{conv}\{\mathbf{e}_i\mid i\in S\}$ is $\prod_{i\in S}x_i$.
\end{example}

\begin{definition}
A labeled cell complex $X$ is a \emph{cellular resolution} of the ideal $I=\langle\ell(v)\mid v\in V(X)\rangle$ if the non-empty complexes $\{\sigma\in X\mid \ell(\sigma)$ divides $m\}$ are acyclic over $\mathbb{K}$ for all monomials $m$.
\end{definition}

\begin{definition}
A cellular resolution is \emph{minimal} if no cell is properly contained in a cell with the same label.
\end{definition}

This definition of minimality implies that the complex obtained from the cell complex is minimal in the algebraic sense of free resolutions. See Remark 1.4 in~\cite{BS} for details.

\begin{example}
The standard simplex $\Delta_n=\mathrm{conv}\{\mathbf{e}_i\mid i\in [n]\}$ labeled in Example~\ref{ex:simp} is a cellular resolution. The complex 
\[\left\{\sigma\in X\mid \ell(\sigma)\textrm{ divides }\prod_{i\in S}x_i\right\}\] is the simplex $\mathrm{conv}\{\mathbf{e}_i\mid i\in S\}$, and it is convex and acyclic. The resolution is minimal as each face has a unique label.
\end{example}

Algebraic discrete Morse theory was developed by Batzies and Welker~\cite{BW}. It provides a way to make non-minimal cellular resolutions smaller. Discrete Morse theory is usually explained in terms of Morse functions, but in the algebraic setting it is more convenient to use acyclic matchings. A good introduction to the general theory of discrete Morse theory is by Forman~\cite{F}, who invented it.

\begin{definition}
A matching $M$ in a directed acyclic graph $D$ is \emph{acyclic} if the directed graph obtained from $D$ by reversing the edges in $M$ is acyclic. An acyclic matching of a poset $F$ is an acyclic matching of its Hasse diagram.
\end{definition}

\begin{definition}
An acyclic matching $M$ of the face poset of a labeled cell complex is \emph{homogeneous} if $\ell(\sigma)=\ell(\tau)$ for any $\sigma,\tau\in M$.
\end{definition}

\begin{definition}
Let $M$ be a matching of the face poset of a labeled cell complex $X$. The cells that are not matched are \emph{critical}.
\end{definition}

The main theorem of algebraic discrete Morse theory for cellular resolutions can be stated as follows.

\begin{theorem}\label{thm:admt}
Let $X$ together with the labeling $\ell$ be a cellular resolution of $I$, and let $M$ be an acyclic homogenous matching of the face poset of $X$. Then there is a cell complex $\tilde{X},$ homotopy equivalent to $X$, whose cells are in bijection with the critical cells of $X$. Moreover, the bijection preserve dimensions of cells, and $\tilde{X}$ with the labeling induced by $\ell$ is a cellular resolution of $I$.
\end{theorem}

The complex $\tilde{X}$ is called the Morse complex of $M$ on $X$. A proof of Theorem~\ref{thm:admt} and other important properties of the Morse complex are in the appendix to~\cite{BW}. The stronger results needed are best explained in terms of gradient paths.

\begin{definition}
Let $M$ be an acyclic matching of the cell complex $X$. A \emph{gradient path} is a directed path in the graph obtained from the Hasse diagram of the face poset of $X$ by reversing the edges in $M$.
\end{definition}

The following is Proposition 7.3 in~\cite{BW}.

\begin{proposition}\label{prop:strong}
Let $M$ be an acyclic matching of the cell complex $X$ and let $\sigma$ and $\tau$ be critical cells. The cell corresponding to $\sigma$ in the Morse complex of $X$ is on the boundary of the cell corresponding to $\tau$ if and only if there is a gradient path from $\sigma$ to $\tau$.
\end{proposition}

In fact the boundary maps can be explicitly described in terms of sums over all gradient paths between the two cells, this is a result by Forman explained in the algebraic setting in Lemma 7.7 in~\cite{BW}. For the matchings in this paper it will turn out that the gradient path between a cell and any of its facets is unique and the resulting complex is regular in the sense of CW-complexes.

\begin{example}\label{ex:sq}
Let $P$ be the square with vertices $(2,0,0),(1,1,0),(1,0,1),$ and $(0,1,1)$. Let $X_P$ be the subdivision of $P$ obtained by cutting with the plane $\{(e_1,e_2,e_3)\in\mathbb{R}^3\mid e_1=1\}$, this complex is depicted in Figure~\ref{fig:subdiv}. The complex $X_P$ turns into a labeled complex by labeling the vertices by $\ell(e_1,e_2,e_3)=x_1^{e_1}x_2^{e_2}x_3^{e_3}$. This labeled complex is a cellular resolution, but it is not minimal. For example the line segment from $(1,1,0)$ to $(1,0,1)$ has the same label as the triangle with vertices $(1,1,0),(1,0,1)$, and $(0,1,1)$. Matching these two cells give a Morse complex isomorphic to the square $P$.
\end{example}

\begin{figure}
\center\includegraphics[width=0.5\textwidth]{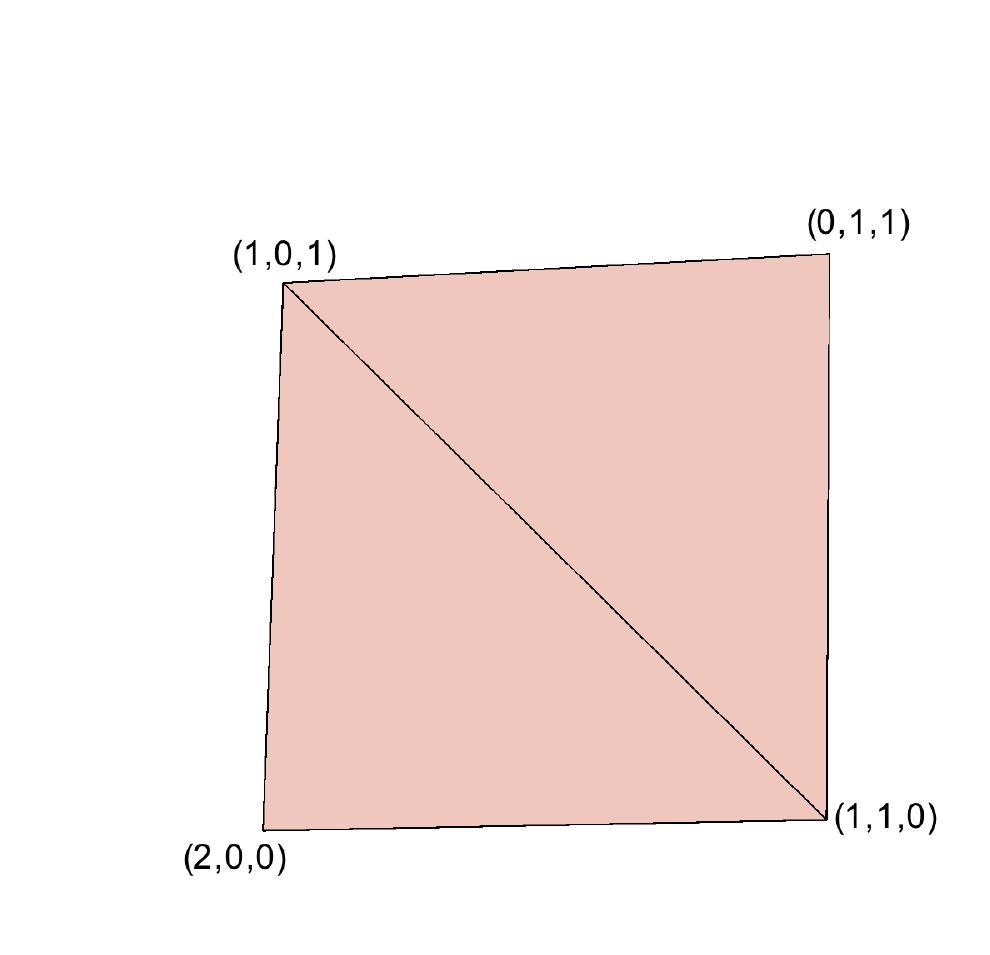}
\caption{The subdivided square in Example~\ref{ex:sq}.}\label{fig:subdiv}
\end{figure}

The following result due to Jonsson, Lemma 4.2 in~\cite{J}, is highly useful when constructing acyclic matchings.

\begin{lemma}\label{lemma:pre}
Let $f$ be a poset map from $P$ to $Q$, and let $M_i$ be an acyclic matching on the preimages $f^{-1}(i)$ for every $i\in Q$. Then the matching $M=\cup_{i\in Q} M_i$ is acyclic.
\end{lemma}

\section{Diced and sharp polytopes}\label{Poly}

This section gives some important definitions regarding polytopes and their subdivisions. In particular we introduce the two new notions of diced and sharp polytopes. Some examples where the subdivisions support cellular resolutions are also provided. Sometimes the subdivisions are trivial in the sense that the only maximal cell is the polytope itself.

\begin{definition}
A polytope $P$ in $\mathbb{R}^n$ is a \emph{lattice polytope} if the vertex set of $P$ is a subset of $\mathbb{Z}^n$.
\end{definition}

\begin{definition}
For a lattice polytope $P$ in $\mathbb{R}^n_{\ge0}$, define the ideal
\[I_P=\langle x_1^{e_1}\cdots x_n^{e_n}\mid (e_1,\ldots,e_n)\in P\cap \mathbb{Z}^n\rangle.\]
\end{definition}

\begin{definition}
Define hyperplanes $H_{i,j}=\{(e_1,\ldots,e_n)\in\mathbb{R}^n\mid e_i=j\}$ and halfspaces $H^{\ge}_{i,j}=\{(e_1,\ldots,e_n)\in\mathbb{R}^n\mid e_i\ge j\}$, $H^{\le}_{i,j}=\{(e_1,\ldots,e_n)\in\mathbb{R}^n\mid e_i\le j\}$ for any integers $i$ and $j$.
\end{definition}

The next definition provides a large class of subdivisions of polytopes that often support cellular resolutions. Later discrete Morse theory will be used to make the resolutions smaller.

\begin{definition}
Let $P$ be a lattice polytope in $\mathbb{R}^n_{\ge0}$. Define $X_P$ to be the subdivision of $P$ obtained by cutting it with the hyperplanes $H_{i,j}$ for all integers $i$ and $j$.
\end{definition}

\begin{definition}
Let $P$ be a lattice polytope in $\mathbb{R}^n_{\ge0}$. Define $O_P$ to be the face poset of cells in $X_P$ that are not contained in the boundary of $P$.
\end{definition}

The maximal cells in $O_P$ are exactly the cells in $X_P$ of the same dimension as $P$.

\begin{example}
Recall the complex $X_P$ in Example~\ref{ex:sq} where $P$ is the square with vertices $(2,0,0),(1,1,0),(1,0,1)$ and $(0,1,1)$.
In this example only one of the hyperplanes was needed to define the complex. The poset $O_P$ consists of two triangles and their common facet.
\end{example}

The following definition is important. Most polytopes will be assumed to satisfy this property.

\begin{definition}
A lattice polytope $P$ in $\mathbb{R}^n_{\ge0}$ is \emph{diced} if $P\cap \mathbb{Z}^n=V(X_P)$.
\end{definition}

\begin{example}\label{ex:lines}
The line from $(2,0)$ to $(0,2)$ is a diced polytope. The vertices of the subdivided complex are $(2,0),(1,1)$ and $(0,2)$. The line from $(0,0)$ to $(1,2)$ is not diced as the vertices of the subdivided complex are $(0,0),(1/2,1)$ and $(1,2)$. The polytopes are depicted in Figure~\ref{fig:lines}.
\end{example}

\begin{figure}
\center\includegraphics[width=\textwidth]{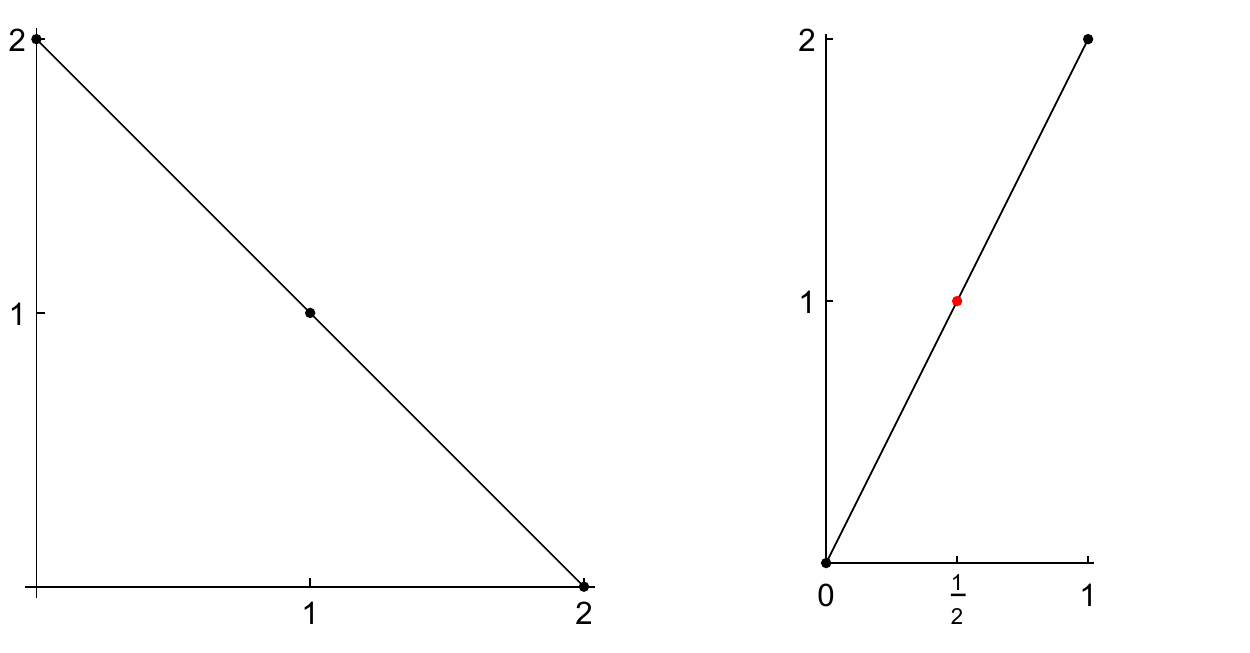}
\caption{To the left is the subdivsion of the diced line from $(2,0)$ to $(0,2)$ in Example~\ref{ex:lines}. To the right is the subdivision of the line from $(0,0)$ to $(1,2)$ with the red vertex $(1/2,1)$ showing that the line is not diced.}\label{fig:lines}
\end{figure}

A rich class of diced polytopes come from totally unimodular matrices.

\begin{definition}
A matrix is \emph{totally unimodular} if all determinats of sub matrices are in $\{-1,0,1\}$. 
\end{definition}

\begin{proposition}\label{prop:uno}
Let $M$ be a totally unimodular matrix in $\mathbb{R}^{m\times n}$ and let $\beta\in\mathbb{Z}^m$. The polytope $P=\{(e_1,\ldots,e_n)\in\mathbb{R}^n_{\ge 0}\mid M\cdot (e_1,\ldots,e_n)^T\le\beta\}$  is diced.
\end{proposition}
\begin{proof}
It is well known that polytopes defined by totally unimodular matrices in this way only have vertices in $\mathbb{Z}^n$. As a consequence if $v$ is a vertex of $X_P$ then there is some hyperplane $H_{i,j}$ containing $v$.

Fixing a coordinate to be a particular integer gives a slice of the polytope. This slice is defined by a smaller still totally unimodular matrix, obtained by deleting a column from $M$. Now the statement follows by induction on $n$.
\end{proof}

\begin{definition}
Let $P$ be a diced polytope. Define $\ell(P)=\mathrm{lcm}\{x_1^{e_1}\cdots x_n^{e_n}\mid (e_1,\ldots,e_n)\in P\cap\mathbb{Z}^n\}$.
\end{definition}

The complex $X_P$ supports a cellular resolution for any diced polytope $P$. Before proving this it is helpful to consider some properties of the complex $X_P$.

A very important property is that all the cells in $X_P$ are lattice polytopes if $P$ is diced. This follows immediately from the definition of diced as all vertices in $X_P$ are in $\mathbb{Z}^n$.

An equivalent way to describe the complex $X_P$, is that $X_P$ is the subdivision of $P$ induced by the subdivision of $\mathbb{R}^n$ into cubes $[e_1-1,e_1]\times\cdots\times[e_n-1,e_n]$ where $(e_1,\ldots,e_n)\in\mathbb{Z}^n$. From this description it follows that each open cell in $X_P$ is contained in a unique half open cube  $(e_1-1,e_1]\times\cdots\times(e_n-1,e_n]$. This is useful when $P$ is diced as the label of a cell can be recovered from the half open cube containing its interior. If the open cell $\sigma$ is contained in $(e_1-1,e_1]\times\cdots\times(e_n-1,e_n]$, then it has to have at least one vertex in each hyperplane $H_{i,e_i}$, as otherwise all vertices would be in some $H_{i,e_i-1}$ and that hyperplane does not intersect the half open cube. This shows that the label of $\sigma$ is $x_1^{e_1}\cdots x_n^{e_n}$.

\begin{proposition}\label{prop:diced}
If $P$ is a diced polytope, then $X_P$ together with $\ell$ is a cellular resolution of $I_P$.
\end{proposition}
\begin{proof}
The preceding discussion showed that an open cell in $X_P$ has label  $x_1^{e_1}\cdots x_n^{e_n}$ if and only if it is contained in the half open cube $(e_1-1,e_1]\times\cdots\times(e_n-1,e_n]$. The label of a cell then divides $x_1^{e_1}\cdots x_n^{e_n}$ if and only if it is contained in $H^{\le}_{1,e_1}\cap\cdots\cap H^{\le}_{n,e_n}$. In particular, the geometric realization of
\[\{\sigma\in X_P\mid \ell(\sigma)\textrm{ divides }x_1^{e_1}\cdots x_n^{e_n}\}\]
is $P\cap H^{\le}_{1,e_1}\cap\cdots\cap H^{\le}_{n,e_n}$. It is convex, and in particular acyclic or empty. We have verified that the complex is a cellular resolution. It resolves $I_P$ as $V(X_P)=P\cap\mathbb{Z}^n.$
\end{proof}

A useful property of these resolutions is that the maximal cells in $O_P$ will have different labels as they are contained in different $[e_1-1,e_1]\times\cdots\times[e_n-1,e_n]$  cubes. To reduce the size of cellular resolutions constructed from diced polytopes we introduce the well behaving sharp ones.

\begin{definition}
A diced polytope $P$ is \emph{sharp} if there is a cell $\sigma_P\in X_P$ so that $\dim \sigma_P=\dim P$ and $\ell(\sigma_P)=\ell(P)$.
\end{definition}

Recall that the maximal elements in the poset $O_P$ are the cells with the same dimension as $P$, in particular if $P$ is sharp then it has the maximal element $\sigma_P$.

\begin{definition}
A diced polytope $P$ is \emph{totally sharp} if all faces of $P$ are sharp.
\end{definition}

\begin{example}
The line segment from $(2,0)$ to $(0,2)$ in Figure~\ref{fig:lines} viewed as a polytope is not sharp. The whole line has label $x_1^2x_2^2$, and the two full dimensional cells have labels $x_1^2x_2$ and $x_1x_2^2$. Observe that the full dimensional cells are sharp.
\end{example}

\begin{example}
The square $P$ in Example~\ref{ex:sq} is sharp. The complex $X_P$ is depicted in Figure~\ref{fig:subdiv}, the lower triangle is the cell $\sigma_P$.
\end{example}

There is a more geometric definition of sharpness.

\begin{proposition}
Let $P$ be a diced polytope with $\ell(P)=x_1^{e_1}\cdots x_n^{e_n}$. Then the polytope $P$ is sharp if and only if $\dim P\cap H^{\ge}_{1,e_1-1}\cap\cdots\cap H^{\ge}_{n,e_n-1}=\dim P$. In this case $\sigma_P=P\cap H^{\ge}_{1,e_1-1}\cap\cdots\cap H^{\ge}_{n,e_n-1}$.
\end{proposition}
\begin{proof}
All maximal cells in $X_P$ are of the form
\[P\cap H^{\ge}_{1,e'_1-1}\cap\cdots\cap H^{\ge}_{n,e'_n-1}\cap H^{\le}_{1,e'_1}\cap\cdots\cap H^{\le}_{n,e'_n}.\]
If $\ell(P)=x_1^{e_1}\cdots x_n^{e_n}$ then $P$ is in $H^{\le}_{1,e_1}\cap\cdots\cap H^{\le}_{n,e_n}$. As in the proof of Proposition~\ref{prop:diced} a cell has label $x_1^{e_1}\cdots x_n^{e_n}$ if and only if the interior of the cell is contained in $(e_1-1,e_1]\times\cdots\times(e_n-1,e_n]$. A cell with label $x_1^{e_1}\cdots x_n^{e_n}$ is contained in $P\cap H^{\ge}_{1,e_1-1}\cap\cdots\cap H^{\ge}_{n,e_n-1}$. If this intersection has the same dimension as $P$ then this cell is $\sigma_P$ and $P$ is sharp. To see this note that the cell cannot be contained in any of the hyperplanes $H_{i,e_{i}-1}$ as $P$ is not. If the intersection has lower dimension, then there is no cell $\sigma_P$ and $P$ is not sharp.
\end{proof}

\section{Subdivisions with sharp cells}\label{Main}

Sharp polytopes work well with respect to algebraic discrete Morse theory.

\begin{lemma}\label{lemma:main}
Let $P$ be a sharp polytope in $\mathbb{R}^n_{\ge0}$. There is a homogeneous acyclic matching of $O_P$ where $\sigma_P$ is the only critical cell.
\end{lemma}
\begin{proof}
The argument is by induction on the number of maximal cells in $O_P$, $\dim P$ and $n$. These are the base cases of the induction:
\begin{itemize}
\item[-] If $O_P$ has a single maximal cell then it is $\sigma_P$, as $\sigma_P$ is always maximal by definition. There are no other maximal cells beside $\sigma_P$ if and only if $\sigma_P=P$. In this case $O_P=\{\sigma_P\}$, and the empty matching leaves only $\sigma_P$ as critical.

\item[-] If $\dim P=0$ then $\sigma_P=P$.

\item[-] If $n=1$, then $\dim P=0$ or $P$ is a line segment. Let $P$ be the line segment $[a,b]$. Any cell $[i-1,i]$ with $i\neq b$ is matched to its endpoint $i$. The only critical cell is $[b-1,b]=\sigma_P$.
\end{itemize}

From here on, it can be assumed that $|O_P|>1,\dim P>0$, and $n>1.$ By induction on $n$, it can also be assumed that $P$ is not contained in any hyperplane $H_{i,j}$.

The maximal elements of $O_P$ are polytopes of the same dimension as $P$. As the subdivision $X_P$ is not trivial, every maximal element is neighboring another maximal element. In particular, there is a maximal element $\tau$ in $O_P\setminus \{\sigma_P\}$ so that $\omega=\tau\cap\sigma_P$ is a facet of $\sigma_P$ and $\tau$.

Let $H_{i_1,j_1},\ldots,H_{i_k,j_k}$ be the hyperplanes of the form $H_{i,j}$ containing $\omega$. As $P$ is not in any hyperplane $H_{i,j}$, the space $H=H_{i_1,j_1}\cap\cdots\cap H_{i_k,j_k}$ is the supporting subspace of $\omega$. Note that while $\omega=\sigma_P\cap H_{i_s,j_s}=\tau\cap H_{i_s,j_s}=\sigma_P\cap H=\tau\cap H$ for any single $i_s\in S=\{i_1,\ldots,i_k\}$ it could happen that $P$ is contained in a hyperplane so that $\omega$ is contained in multiple hyperplanes $H_{i,j}$.

The space $H$ splits $P$ into two polytopes $P_{\ge}$ and $P_{\le}$. The polytope $P_{\ge}$ contains $\sigma_P$ and $P_{\le}$ contains $\tau$. In fact
\[P_{\ge}=P\cap H^{\ge}_{i_1,j_1}\cap\cdots\cap H^{\ge}_{i_k,j_k}\]
and
\[P_{\le}=P\cap H^{\le}_{i_1,j_1}\cap\cdots\cap H^{\le}_{i_k,j_k}\]
as otherwise one of the variables would have a higher degree in $\ell(\tau)$ than in $\ell(\sigma_P)$.

Let $\ell(P)=x_1^{e_1}\cdots x_n^{e_n}$ and recall that $H=H_{i_1,j_1}\cap\cdots\cap H_{i_k,j_k}$ and $S=\{i_1,\ldots,i_k\}$. Now $j_s=e_{i_s}-1$ for all $s\in[k]$ and equivalently $H_{i_s,j_s}=H_{i_s,e_{i_s}-1}$ for all $i_s\in S$.

The next step is to show that $P_{\ge},P_{\le},$ and $P\cap H$ are all sharp.

By construction $\sigma_P\subseteq P_{\ge}\subseteq P$ and it follows that $P_{\ge}$ is sharp with $\sigma_{P_{\ge}}=\sigma_P$.

The polytope $P_{\le}$ is contained in $(\cap_{i\in S}  H^{\le}_{i,e_i-1})\cap(\cap_{i\in [n]\setminus S} H^{\le}_{i,e_i})$ and in fact the cell $\tau$ can be described explicitly as
\[\tau=P_{\le}\cap (\cap_{i\in S}  H^{\ge}_{i,e_i-2})\cap(\cap_{i\in [n]\setminus S} H^{\ge}_{i,e_i-1}).\]
In particular $P_{\le}$ is sharp with $\tau=\sigma_{P_{\le}}$.

The polytope $P\cap H$ is also contained in $(\cap_{i\in S}  H^{\le}_{i,e_i-1})\cap(\cap_{i\in [n]\setminus S} H^{\le}_{i,e_i})$ and
\[\omega=P\cap H\cap (\cap_{i\in S}  H^{\ge}_{i,e_i-2})\cap(\cap_{i\in [n]\setminus S} H^{\ge}_{i,e_i-1})\]
is full dimensional in $P\cap H$. The polytope $P\cap H$ is sharp with $\sigma_{P\cap H}=\omega$ and $\ell(\omega)=\ell(\tau)=\prod_{i\in S}x_i^{e_i-1}\prod_{i\in[n]\setminus S}x_i^{e_i}$.

The posets $O_{P_{\ge}}$ and $O_{P_{\le}}$ has fewer maximal elements than $O_P$, and $P\cap H$ has lower dimension than $P$. By induction there are acyclic matchings $M_{P_{\ge}},M_{P_{\le}}$ and $M_{P\cap H}$ of $O_{P_{\ge}},O_{P_{\le}}$ and $O_{P\cap H}$ respectively, leaving only $\sigma_{P},\tau,$ and $\omega$ as critical. The matching $M_{P_\ge}\cup M_{P_{\le}}\cup M_{P\cap H}$ is acyclic by Lemma~\ref{lemma:pre}. The poset map used is constructed as follows. Let $Q$ be the polytopes $P_\ge,P_{\le}$ and $P\cap H$ ordered by containment, and the poset map sends a cell $\sigma$ to the smallest polytope in $Q$ containing $\sigma$. The preimage of $P_\ge,P_{\le}$ and $P\cap H$ are $O_{P_{\ge}},O_{P_{\le}}$ and $O_{P\cap H}$, respectively. Adding $\tau\omega$ to the matching does not break acyclicity. To see this consider the poset map to the labels ordered by divisibility. There is no other maximal cell with the same label as $\tau$, and then there can be no cycle using the reversed edge between $\omega$ and $\tau$. Now $M_{P_\ge}\cup M_{P_{\le}}\cup M_{P\cap H}\cup \{\tau\omega\}$ is a homogenous acyclic matching leaving only $\sigma_P$ as critical.
\end{proof}

\begin{theorem}\label{thm:maint}
If $X$ is any subdivision of a diced polytope $P$ into totally sharp polytopes, then $X$ together with $\ell$ is a cellular resolution of $I_P$.
\end{theorem}
\begin{proof}
Subdivide $X$ further by all hyperplanes $H_{i,j}$ to obtain a complex $X'$. As all cells in $X$ are diced it follows that all cells in $X'$ are lattice polytopes. Arguing as in the proof of Proposition~\ref{prop:diced}, the geometric realization of 
\[\{\sigma\in X'\mid \ell(\sigma)\textrm{ divides }x_1^{e_1}\cdots x_n^{e_n}\}\]
is $P\cap H^{\le}_{1,e_1}\cap\cdots\cap H^{\le}_{n,e_n}$, and the complex give a cellular resolution of $I_P$.

For each cell $\sigma$ in $X$, Lemma~\ref{lemma:main} provides a matching of the cells in $X'$ contained in $\sigma$. Lemma~\ref{lemma:pre} shows that the matchings glue together to a matching for all of $X'$ where the critical cells are in bijection with the cells of $X$.

The bijection preserve both dimension and label. To show that the Morse complex in fact is isomorphic to $X$ a slight strengthening of Theorem~\ref{thm:admt} is needed. Proposition~\ref{prop:strong} ensures that a cell $\sigma$ is on the boundary of a cell $\tau$ in the Morse complex if and only if the corresponding faces in the polytope $\sigma'$ and $\tau'$ satisfy $\sigma'\subset\tau'$. It is enough to consider the case when $\sigma'$ is a facet of $\tau'$, in this case the gradient path can be constructed inductively by observing that the last step has to come from the interior of the cell in $X$ containing $\sigma$. This inductive argument also shows that the gradient path is unique. Note that the gradient paths are not unique if looking at faces of higher codimension.

Proposition 7.7 in~\cite{BW} describes the boundary maps of the Morse complex in terms of sums over gradient paths between a cell and a codimension one cell on its boundary. As these gradient paths are unique it follows that the Morse complex is regular in the sense of CW-complexes and isomorphic to $X$. Essentially this also follows from inductively using the proof of Theorem 12.1 in~\cite{F}.
\end{proof}

The Minkowski sum of standard simplices $P_1,\cdots,P_m$ is the polytope $P=P_1+\cdots+P_m=\{p_1+\cdots+p_m\mid p_i\in P_i\}$. A fine mixed subdivision of the Minkowski sum $P$ is a subdivision of $P$ into polytopes $B_1+\cdots+B_m$ where each $B_i$ is a simplex and a facet of $P_i$ and furthermore the $B_i$ are in affinely independent subspaces.

\begin{corollary}\label{thm:fine}
Fine mixed subdivisions of Minkowski sums of standard simplices are minimal cellular resolutions.
\end{corollary}
\begin{proof}
To show that it is a resolution it is enough to check that the cells are totally sharp.

Corollary 4.9 in~\cite{OY} says that the matrices defining fine mixed cells are totally unimodular, in particular the cells are diced by Proposition~\ref{prop:uno}.

Let $\Delta_1,\ldots,\Delta_N$ be the simplices so that $\Delta_1+\cdots+\Delta_N$ is a fine mixed cell. The label of the cell $\Delta_1+\cdots+\Delta_N$ is $x_1^{e_1}\cdots x_n^{e_n}$ where $e_i$ is the number of simplices containing the standard basis vector $\mathbf{e}_i$. It can be assumed that $e_i>0$ for all $i$. Lemma 2.6 in~\cite{OY} says that a the cell contains the simplex $(e_1-1,\ldots,e_n-1)+\mathrm{conv}\{\mathbf{e}_1,\ldots,\mathbf{e}_n\}$, this shows that it is sharp.

Minimality follow from the fact that the degree of $\ell(\Delta_1+\cdots+\Delta_N)$ is $|V(\Delta_1)|+\cdots+|V(\Delta_N)|$ and all facets of $P$ have label of lower degree.
\end{proof}

A three-dimensional fine mixed cell is either a simplex, a triangular prism, or a cube. To illustrate the results, an example of each type is examined in more detail.

\begin{example}
If the cell $P$ is a simplex, then $P=\sigma_P$ and no matching is needed.
\end{example}

\begin{example}\label{ex:prism}
Let $P$ be the triangular prism $\mathrm{conv}\{\mathbf{e}_1,\mathbf{e}_2,\mathbf{e}_3\}+\mathrm{conv}\{\mathbf{e}_1+\mathbf{e}_4\}$. The complex $X_P$ is depicted in Figure~\ref{fig:prism}. The simplex is the cell $\sigma_P$. The interior triangle is matched to the pyramid. The subdivided cells on the boundary are isomorphic to the  complex in Example~\ref{ex:sq}, which is depicted in Figure~\ref{fig:subdiv}.
\end{example}

\begin{figure}
\center\includegraphics[width=\textwidth]{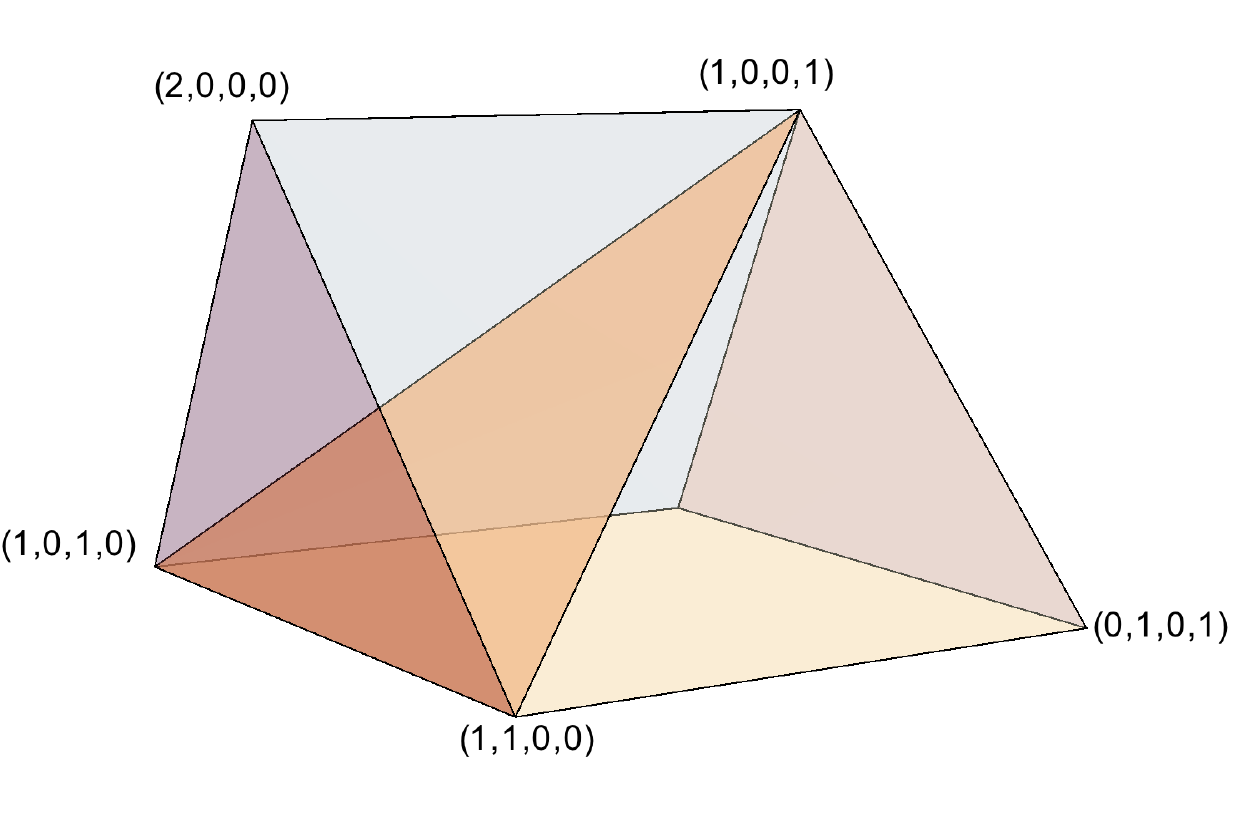}
\caption{The subdivided triangular prism in Example~\ref{ex:prism}. The unlabeled vertex is $(0,0,1,1)$. The hyperplane cutting the prism is $H_{1,1}$.}\label{fig:prism}
\end{figure}

\begin{example}\label{ex:cube}
Let $P$ be the cube $\mathrm{conv}\{\mathbf{e}_1,\mathbf{e}_2\}+\mathrm{conv}\{\mathbf{e}_1,\mathbf{e}_3\}+\mathrm{conv}\{\mathbf{e}_3,\mathbf{e}_4\}$. The complex $X_P$ is depicted in Figure~\ref{fig:cube}. The rightmost simplex is the cell $\sigma_P$. The interior triangles on the boundary of $\sigma_P$ are matched to the pyramids. Any of the remaining two interior triangles can be matched to the leftmost simplex. The line from $(1,0,1,1)$ to $(1,1,1,0)$ can be matched to the last remaining interior triangle. As in Example~\ref{ex:prism} the subdivided cells on the boundary are handled like the complex in Figure~\ref{fig:subdiv}.
\end{example}

\begin{figure}
\center\includegraphics[width=\textwidth]{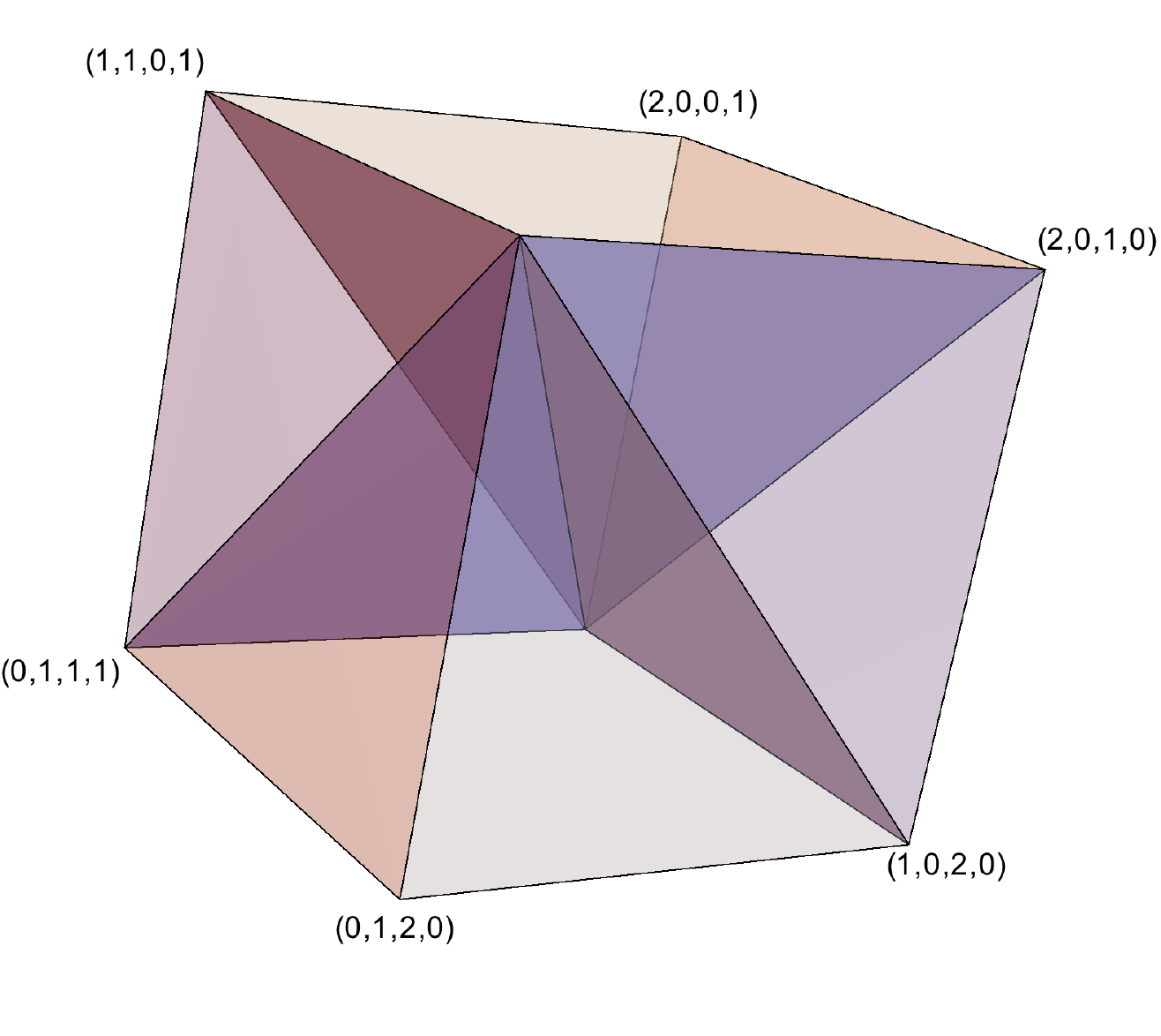}
\caption{The subdivided cube in Example~\ref{ex:cube}. The unlabeled vertex at the back is $(1,0,1,1)$ and the one in front is $(1,1,1,0)$. The hyperplanes cutting the cube are $H_{1,1}$ and $H_{3,1}$.}\label{fig:cube}
\end{figure}

There are many more resolutions of this class.

\begin{proposition}
All zero one polytopes are diced and totally sharp.
\end{proposition}
\begin{proof}
If $P$ is a zero one polytope then the subdivision $X_P$ is trivial and $P$ is diced. For all faces $F$ it also hold that $\sigma_F=F$ and $P$ is totally sharp.
\end{proof}

The resolutions coming from the zero one polytopes are the hull resolutions of square-free ideals. These resolutions can sometimes be made minimal using discrete Morse theory but in general the Morse complex is no longer a polytopes, and sometimes it is impossible to make the resolution minimal using only discrete Morse theory \cite{N}.

Another polytopal subdivision where the cells are totally sharp occur in~\cite{EN}, where the resolution resolves a power of the edge ideal of a path. It is interesting to note that if a not subdivided diced polytope support a minimal resolution then it has to be sharp.
 
\begin{proposition}
Let $P$ be a diced polytope. If the not subdivided polytope $P$ together with $\ell$ supports a minimal cellular resolution of $I_P$ then $P$ is sharp.
\end{proposition}
\begin{proof}
The subdivision $X_P$ also supports a resolution. As $P$ supports a minimal resolution and has a cell of dimension $\dim P$ with label $\ell(P)$ then $X_P$ also needs such a cell, this cell has to be $\sigma_P$ showing the sharpness of $P$.
\end{proof}

\section*{Acknowledgements}

Thanks to Alex Engstr\"om and Anton Dochtermann for helpful comments on an early version of the manuscript.

\end{document}